\documentclass[preprint,1p]{elsarticle}

\makeatletter
 \def\ps@pprintTitle{%
 	\let\@oddhead\@empty
 	\let\@evenhead\@empty
 	\def\@oddfoot{\footnotesize\itshape
 		{} \hfill\today}%
 	\let\@evenfoot\@oddfoot
 }
\makeatother
\usepackage[unicode]{hyperref}

\usepackage{latexsym}
\usepackage{indentfirst}
\usepackage{amsxtra}
\usepackage{amssymb}
\usepackage{amsthm}
\usepackage{amsmath}
\usepackage{mathrsfs} 

\usepackage{xcolor}
\usepackage{color}

\usepackage{amsfonts}

\usepackage[capitalise]{cleveref}

\newtheorem{theor}{Theorem}
\newtheorem*{theor*}{Theorem}
\newtheorem{prop}[theor]{Proposition}
\newtheorem{lemma}[theor]{Lemma}
\newtheorem{cor}[theor]{Corollary}
\newtheorem*{cor*}{Corollary}
\theoremstyle{definition}               
\newtheorem{defin}[theor]{Definition}
\newtheorem{ex}{Example}
\newtheorem{exs}[ex]{Examples}

\newtheorem{que}{Question}

\DeclareMathOperator{\End}{End}
\DeclareMathOperator{\id}{id}

\DeclareMathOperator{\tr}{tr}
\DeclareMathOperator{\Ker}{Ker}

\DeclareMathOperator{\E}{E}

\usepackage[makeroom]{cancel}
 
\usepackage{soul}
\usepackage{tikz}

\begin{document}

\begin{frontmatter}
	\title{Set-theoretical solutions of the pentagon equation\\ on Clifford semigroups\tnoteref{mytitlenote}}

 	\tnotetext[mytitlenote]{This work was partially supported by the Dipartimento di Matematica e Fisica ``Ennio De Giorgi'' - Università del Salento and the Departament de Matemàtiques - Universitat de València. The first and the third authors are members of GNSAGA (INdAM) and of the non-profit association ADV-AGTA.}
 		\author[unile]{Marzia~MAZZOTTA}
	\ead{marzia.mazzotta@unisalento.it}
	\author[unied]{Vicent P\'EREZ-CALABUIG}
	\ead{vicent.perez-calabuig@uv.es}
	\author[unile]{Paola~STEFANELLI}
 	\ead{paola.stefanelli@unisalento.it}
 	\address[unile]{Dipartimento di Matematica e Fisica ``Ennio De Giorgi'',
 		\\
		Universit\`{a} del Salento,\\
 		Via Provinciale Lecce-Arnesano, \\
 		73100 Lecce (Italy)\\}
 		\address[unied]{Departament de Matemàtiques de València,\\ Dr. Moliner, 50, \\46100 Burjassot, València (Spain)}

\begin{abstract} 
Given a set-theoretical solution of the pentagon equation $s:S\times S\to S\times S$ on a set $S$ and writing $s(a, b)=(a\cdot b,\, \theta_a(b))$, with $\cdot$ a binary operation on $S$ and $\theta_a$ a map from $S$ into itself, for every $a\in S$, one naturally obtains that $\left(S,\,\cdot\right)$ is a semigroup.\\ 
In this paper, we focus on solutions on Clifford semigroups $\left(S,\,\cdot\right)$ satisfying special properties on the set of the idempotents $\E(S)$. 
Into the specific, we provide a complete description of \emph{idempotent-invariant solutions}, namely, those solutions for which $\theta_a$ remains invariant in $\E(S)$, for every $a\in S$. 
Moreover, considering $(S,\,\cdot)$ as a disjoint union of groups, we construct a family of \emph{idempotent-fixed solutions},  i.e., those solutions for which $\theta_a$ fixes every element in $\E(S)$, for every $a\in S$, starting from a solution on each group.
\end{abstract}

\end{frontmatter}

\section*{Introduction}
If $V$ is a vector space over a field $F$, a linear map $\mathcal{S}: V \otimes V \to V \otimes V$ is said to be a \emph{solution of the pentagon equation} on $V$ if it satisfies the relation
	\begin{align}\label{vec_equa}
	\mathcal{S}_{12}\mathcal{S}_{13}\mathcal{S}_{23}=\mathcal{S}_{23}\mathcal{S}_{12},
	\end{align}
	where 
		$\mathcal{S}_{12}=\mathcal{S}\otimes \id_V$,\,$\mathcal{S}_{23}=\id_V\otimes \, \mathcal{S}$,\,$\mathcal{S}_{13}=(\id_V\otimes\, \Sigma)\,\mathcal{S}_{12}\;(\id_V\otimes \, \Sigma)$,
	with $\Sigma$ the flip operator on $V\otimes V$, i.e., $\Sigma(u\otimes v)=v\otimes u$, for all $u,v\in V$. 
    The pentagon equation arose at first at the beginning of '80 in \cite{Bi81} as the Biedenharn-Elliott identity for Wigner $6j-$symbols and Racah coefficients in the representation theory for the rotation group. Maillet \cite{Ma94} showed that solutions of the pentagon equation lead to solutions of the tetrahedron equation \cite{Za80}, a generalization of the well-known quantum Yang-Baxter equation \cite{Ya67, Ba72}. Moreover, in \cite[Theorem 3.2]{Mi04}, Militaru showed that bijective solutions on finite vector spaces are equivalent to finite Hopf algebras, and so the classification of the latter is reduced to the classification of solutions.
    In the subsequent years, the pentagon equation appeared in literature in several forms with
    different terminologies according to the specific research areas. We highlight some interesting works as \cite{Dr89,    Za92,   MaNaYo94, VDVK94, Ka96, Wo96,  BaSk98, Str98, Ka99, BaSk03, Mi04, JiLi05, Fu10}, just to name a few.
    For a fuller treatment of some applications in which the pentagon equation appears, we suggest the recent paper by Dimakis and M\"uller-Hoissen \cite{DiMu15} (along with the references therein), where the authors dealt with an infinite family of equations named \emph{polygon equations}.
    
    As well as Drinfel'd in \cite{Dr92} translated the study of solutions of the Yang-Baxter equation into set-theoretical terms, Kashaev and Sergeev in \cite{KaSe98} began  the study of the pentagon equation with a set-theoretical approach. Namely, if $S$ is a set, a map $s:S\times S\to S\times S$ satisfying the following ``reversed" relation
 \begin{equation}\label{set_equa}
	s_{23}s_{13}s_{12}=s_{12}s_{23},
\end{equation}
where $s_{12}=s\times \id_S$,  $s_{23}=\id_S\times s$, $s_{13}=(\id_S\times \tau)\,s_{12}\,(\id_S\times \tau)$, and $\tau(a,b)=(b,a)$, for all $a,b\in S$,
is said to be a \emph{set-theoretical solution of the pentagon equation}, or briefly \emph{solution}, on $S$. If, in particular, $s$ is a solution on a finite set $S$, then the linear map $\mathcal{S}:F^{S \times S} \to F^{S \times S}$ defined by $\mathcal{S}(f)(a,b)= f(s(a,b))$, for all $a,b \in S$, is a solution of \eqref{vec_equa} on the vector space $F^S$ of all functions from $S$ to $F$.\\
For their purposes, the authors in \cite{KaSe98} investigated only bijective maps. This  class of solutions was also studied by Kashaev and Reshetikhin in \cite{KaRe07}, where it is shown that each symmetrically factorizable Lie group is related to a bijective solution. Among these solutions, a description of all those that are involutive, i.e., $s^2=\id_{S\times S}$, has been recently given by Colazzo, Jespers, and Kubat in \cite{CoJeKu20}. \\
As one can see in \cite[Proposition 8]{CaMaMi19}, any arbitrary solution $s$ on a set $S$ can be written as $s(a, b)=(a\cdot b,\, \theta_a(b))$, with $\cdot$ a binary operation on $S$ and $\theta_a$ a map from $S$ into itself, for every $a\in S$.  In this way, $S$ is inherently endowed with a structure of a semigroup $\left(S,\,\cdot\right)$ and it appears natural the study of solutions on specific classes of semigroups. For brevity, we will denote the multiplication in $S$ as a concatenation. In this vein, in \cite[Theorem 15]{CaMaMi19} the authors provide a description of all solutions on a group, by showing that they are determined by its normal subgroups. Moreover, in \cite{CaMaSt20}, we can find several constructions of solutions on semigroups, such as on the matched product of two semigroups, that is a semigroup including the classical Zappa-Sz\'ep product. In the same paper, the authors investigate maps that are both solutions of the pentagon and the Yang-Baxter equations \cite{Dr92}. Furthermore, in \cite{Maz23x}, the first author study the idempotent solutions, namely, maps satisfying the property $s^2=s$, and describes this kind of solutions on monoids having central idempotents.

In this paper, we begin the study of solutions on Clifford semigroups, namely, inverse semigroups whose idempotent elements are central. Recalling that a semigroup $S$ is inverse if, for each 
$a\in S$, there exists a unique $a^{-1}\in S$ 
satisfying $aa^{-1}a = a$ and $a^{-1}aa^{-1} = a^{-1}$, 
it is clear that the behaviour of Clifford semigroups is very close to that of groups. 
In light of this fact and the description of solutions on groups in \cite{CaMaMi19}, it is natural to wonder if a description of solutions can be obtained also on this class of semigroups. However, such an aim appears challenging and some considerations on the set of solutions must be considered. It is easy to check that every solution on a group $G$ satisfies that $\theta_a(1) = 1$, for every $a\in G$. Therefore, it motivates us to  consider both classes of solutions on a Clifford semigroup $S$ such that $\theta_a$, respectively, fixes every idempotent or remains invariant on every idempotent, for every $a\in S$. We call them, respectively, \emph{idempotent-fixed} and \emph{idempotent-invariant} solutions. \\
The main results of this paper are the following. Firstly, we provide a complete description of the first class of solutions on a Clifford semigroup $S$, which includes that made in the context of groups. To this aim, we introduce the \emph{kernel} of an arbitrary solution on $S$, which turns out to be a normal subsemigroup, that is a subsemigroup containing the idempotents and closed by conjugation. Secondly, for the second class, considering that any Clifford semigroup is a union of a family of pairwise disjoint groups $\{G_e\}_{e\in \E(S)}$, we give a construction of solutions obtained starting from a solution on each group $G_e$.

\bigskip

\section{Preliminaries}
The aim of this section is to briefly introduce some basics of set-theoretical solutions of the pentagon equation. Initially, we recall some notions related to Clifford semigroups useful for our purposes. For a fuller treatment of this topic, we refer the reader to \cite{ClPr61} and \cite{La98}.

\subsection{Basics on Clifford semigroups}

Recall that $S$ is an \emph{inverse semigroup} if for each $a\in S$ there exists a unique $a^{-1}\in S$ such that $a = aa^{-1}a$ and $a^{-1} = a^{-1}aa^{-1}$. They hold $(a b)^{-1}=b^{-1} a^{-1}$ and $(a^{-1})^{-1}=a$, for all $a,b \in S$. Moreover, $\E(S) = \{\,aa^{-1} \ | \ a\in S\,\} 
   = \{\,a^{-1}a \ | \ a\in S\,\}$
and one can consider the following natural partial order relation
\begin{align*}
\forall \ e,f\in \E(S)\qquad e\leq f \ \Longleftrightarrow \ e = ef = fe.
\end{align*}
An inverse semigroup $S$ is \emph{Clifford} if $aa^{-1}=a^{-1}a$, for any $a\in S$, or, equivalently, the idempotents are central in the sense that commute with every element in $S$.

Given a Clifford semigroup $S$, we introduce the following relations and the properties involved themselves. They are an easy  consequence of the fact that all Green's relations coincide in $S$ and they characterize the structure of $S$ itself. If $a,b\in S$, we define
\begin{enumerate}
\item $a \leq b$ if, and only if, $aa^{-1}\leq bb^{-1}$, which is an extension of the natural partial order in $S$;
\item $a \, \mathcal{R} \, b$ if, and only if, $a \leq b$ and $b \leq a$.
\end{enumerate}
It follows that $\leq$ is a preorder on $S$  and $\mathcal{R}$ is an equivalence relation on $S$ such that  $$G_{aa^{-1}}:=[a]_\mathcal{R} = \{b \in S\, \mid \, bb^{-1} = aa^{-1}\}$$ is a group with identity $aa^{-1}$, for every $a \in S$.  On the other hand, for all $a, b \in S$, 
\begin{align}
\label{leq-prod}
a \leq b \,\Longleftrightarrow\, \exists\, u \in S\quad a=ub \,\,\vee \,\, a=bu.
\end{align} 
Moreover, $\leq$ induces an order relation on the equivalence classes of $\mathcal{R}$,  namely, for all $e,f\in \E(S)$, $G_e \leq G_f$ if, and only if, $e \leq f$. The following theorem describes Clifford semigroups.
\begin{theor}\label{theo_Clifford}
Let $S$ be a Clifford semigroup. Then, 
\begin{enumerate}
    \item $S$ is a union of a family of pairwise disjoint groups $\{G_e\}_{e\in \E(S)}$;
    \item the map $\varphi_{f,e} \colon G_f \rightarrow G_e$ given by $\varphi_{f,e}(b) = eb$, for every $b \in G_f$, is a group homomorphism, for all $e,f\in \E(S)$ such that $e \leq f$;
    \item for all $e,f,g \in \E(S)$ such that $e \leq f \leq g$, then $\varphi_{g,e} = \varphi_{f,e}\varphi_{g,f}$.
\end{enumerate}
\end{theor}
As a consequence of the previous theorem, the product in Clifford semigroups can be written by means of the group homomorphisms $\varphi_{e,f}$, namely,
\begin{align*}
    ab = (ae)(fb) = (efa)(efb) =  \varphi_{e, ef}\left(a\right)\varphi_{f, ef}\left(b\right)\in G_{ef},
\end{align*}
for all $a\in G_e$, $b\in G_f$. In particular, for all $a \in S$, $e \in \E(S)$ such that $a \leq  e$, then $ae = ea = a$.

\noindent For the sake of completeness, the converse of \cref{theo_Clifford} is also true.

\medskip

\subsection{Basics on solutions}

Kashaev and Sergeev \cite{KaSe98} first dealt with solutions from an algebraic point of view. Recently, the study of these solutions has been recovered in \cite{CaMaMi19, CaMaSt20, CoJeKu20, Maz23x}. Following the notation introduced in these works, given a set $S$ and a map $s$ from $S \times S$ into itself, we will write
\begin{center}
	$s(a,b):=(a b, \,\theta_a(b))$,
\end{center}
for all $a,b \in S$, where $\theta_a$ is a map from $S$ into itself, for every $a \in S$. 
Then, $s$ is briefly a \emph{solution} on $S$ if, and only if, the following conditions hold
	\begin{align}
 (ab)c&=a(bc) \notag\\
	\theta_a(b)  \theta_{a b}(c)&=\theta_a(b c)\label{eq:P1}\tag{P1}\\\
\theta_{\theta_a(b)}\theta_{a  b}&=\theta_b\label{eq:P2}\tag{P2}
	\end{align}
for all $a,b,c \in S$. Thus, the first identity naturally gives rise to a semigroup structure on $S$, which leads the study of solutions to focus on specific classes of semigroups. When describing solutions, it serves to distinguish those solutions that are not isomorphic. 
\begin{defin}
    Let $S, T$ be two semigroups and $s(a,b)=(ab, \theta_a(b))$, $t(u, v)=(uv, \eta_u(v))$ two solutions on $S$ and $T$, respectively. Then, $s$ and $t$ are \emph{isomorphic} if there exists an isomorphism $\psi: S \to T$ such that
    \begin{align}\label{isomor}
        \psi\theta_a(b)=\eta_{f\left(a\right)}\psi(b),
    \end{align}
    for all $a,b \in S$, or, equivalently, $(\psi \times \psi)s=t( \psi\times \psi)$.    \end{defin}

\medskip

The following are easy examples of solutions used throughout this paper.
\begin{exs} \hspace{1mm} \label{exs_solutions}
    \begin{enumerate}
        \item  Let $S$ be a set and $f, g:S\to S$ idempotent maps such that $fg=gf$. Then, $s(a,b) = \left(f\left(a\right),\, g\left(b\right)\right)$ is a solution on $S$ (cf. \cite{Mi98}).
 \vspace{1mm}
    \item Let $S$ be a semigroup and $\gamma\in \End(S)$ such that $\gamma^2=\gamma$. Then, the map $s$ given by
	$s(a,b)=\left(ab,\gamma\left(b\right)\right),$
	for all $a,b \in S$, is a solution on $S$ (see \cite[Examples 2-2.]{CaMaMi19}). 
\end{enumerate}
\end{exs}

\medskip
\noindent Let us observe that every Clifford semigroup $S$ gives rise to the following solutions 
\begin{align}\label{solu_Clifford}
\mathcal{I}(a,b)=(ab, b), \qquad \mathcal{F}(a,b)=\left(ab, bb^{-1}\right), \qquad \mathcal{E}(a,b)=(ab, e), 
\end{align}
where $e \in \E(S)$ is a fixed idempotent of $S$, belonging to the class of solutions in $2.$ of \cref{exs_solutions}. 
\medskip

In \cite{BaSk93}, solutions of \eqref{vec_equa} are defined on Hilbert spaces in terms of commutative and cocommutative multiplicative unitary operators (see \cite[Definition 2.1]{BaSk93}). These operators motivate the following classes of solutions in the set-theoretical case.

\begin{defin}
	A solution $s: S \times S \to S \times S$ is said to be \emph{commutative} if $s_{12}s_{13}=s_{13}s_{12}$ and  \emph{cocommutative} if $s_{13}s_{23}=s_{23}s_{13}$.
\end{defin}

\noindent Solutions in \cref{exs_solutions}-$1.$ are both commutative and cocommutative. In \cite[Corollary 3.4]{CoJeKu20}, it is proved that if $s$ is an involutive solution, i.e., $s^2=\id_{S \times S}$, then $s$ is both commutative and cocommutative.

\medskip

\noindent\textbf{Convention}: In the sequel, we assume that 
$S$ is a Clifford semigroup and simply write that $s$ is a solution on $S$ instead of $s(a,b) = (ab, \theta_a(b))$, for all $a,b\in S$.
\bigskip

\section{Properties of solutions on Clifford semigroups}
In this section, we show the existence of a normal subsemigroup associated to any solution $s$ on $S$.
We point out that the properties we proved are consistent with those given in the context of groups \cite{CaMaMi19}. 
\medskip
\begin{prop}\label{prop_properties_theta}
Let $s$ be a solution on $S$. Then, the following statements hold:
\begin{enumerate}
    \item $\theta_a\left(a^{-1}\right)=\theta_{aa^{-1}}\left(a\right)^{-1}$,
    \item $\theta_a\left(a^{-1}a\right)=\theta_a\left(a^{-1}\right)\theta_a\left(a^{-1}\right)^{-1}\in \E(S)$,
    \item $\theta_{aa^{-1}}=\theta_{\theta_{a^{-1}}\left(aa^{-1}\right)}\theta_{a^{-1}}$,
\end{enumerate}
    for every $a \in S$.
    \begin{proof}
    Let $a \in S$. Then, by \eqref{eq:P1}, we have
    \begin{align*}
        \theta_a\left(a^{-1}\right)\theta_{aa^{-1}}\left(a\right)\theta_a\left(a^{-1}\right)&=\theta_a\left(a^{-1}a\right)\theta_a\left(a^{-1}\right)=\theta_a\left(a^{-1}a\right)\theta_{aa^{-1}a}\left(a^{-1}\right)\\
        &=\theta_a\left(a^{-1}aa^{-1}\right)=\theta_a\left(a^{-1}\right)
    \end{align*}
    and $\theta_{aa^{-1}}\left(a\right)\theta_a\left(a^{-1}\right)\theta_{aa^{-1}}\left(a\right)=\theta_{aa^{-1}}\left(aa^{-1}\right)\theta_{aa^{-1}}\left(a\right)=\theta_{aa^{-1}}\left(aa^{-1}a\right)=\theta_{aa^{-1}}\left(a\right),
    $
    hence $\theta_a\left(a^{-1}\right)=\theta_{aa^{-1}}\left(a\right)^{-1}$, so $1.$ is satisfied.\\
    Moreover, by $1.$, we get $\theta_a\left(a^{-1}a\right)=\theta_a\left(a^{-1}\right)\theta_{aa^{-1}}\left(a\right)=\theta_a\left(a^{-1}\right)\theta_a\left(a^{-1}\right)^{-1}$,
    thus $\theta_a\left(a^{-1}a\right)$ is an idempotent of $S$.\\
    Finally, by \eqref{eq:P2}, $\theta_{aa^{-1}}=\theta_{\theta_{a^{-1}}\left(aa^{-1}\right)}\theta_{a^{-1}aa^{-1}}=\theta_{\theta_{a^{-1}}\left(aa^{-1}\right)}\theta_{a^{-1}}$,
    which is our claim.
    \end{proof}
\end{prop}

\noindent Note that the previous result also holds in any inverse semigroup that is not necessarily Clifford.
\medskip

Now, let us introduce a crucial object in studying solutions on Clifford semigroups. 
\begin{defin}
If $s$ is a solution on $S$, the following set
\begin{align*}
    K =\{a\in S\, \mid \, \forall \ e \in \E(S), \,\, e \leq a\quad \theta_e(a) \in \E(S)\}
\end{align*}
is called the \emph{kernel} of $s$.
\end{defin}

\noindent Consistently with \cite[Lemma 13]{CaMaMi19}, our aim is to show that $K$ is a \emph{normal subsemigroup} of the Clifford $S$, namely, $\E(S)\subseteq K$ and $a^{-1}Ka \subseteq K$, for every $a \in S$. To this end, we first provide a preliminary result.

\begin{lemma}\label{lemma_vicent}
Let $s$ be a solution on $S$ and $K$ the kernel of $s$. Then, they hold:
\begin{enumerate}
    \item $\theta_a(e) \in \E(S)$, for all $a\in S$ and $e\in \E(S)$ such that  
    $a\leq e$;
    \item $\theta_{ea}(k) \in \E(S)$, for all $a \in S$, $k \in K$, and $e \in \E(S)$ such that $e \leq a$, $e \leq k$.
\end{enumerate}
\begin{proof}
Let $a \in S$ and $e \in \E(S)$. If $a \leq e$, by \eqref{eq:P1}, we obtain $\theta_a(e) = \theta_a(e)\theta_{ae}(e) = \theta_a(e)^2$, hence $1.$ follows.
Now, if $k \in K$ and $e \leq a$, $e \leq k$, then $\theta_e(k) \in \E(S)$ and, by \eqref{eq:P2}, 
\begin{align*}
  \theta_{ea}(k)=\theta_{\theta_{a^{-1}}\left(ea\right)}\theta_{a^{-1}ea}(k)=\theta_{\theta_{a^{-1}}\left(ea\right)}\theta_{e}(k).
\end{align*}
If we prove that $\theta_{a^{-1}}\left(ea\right) \leq \theta_{e}(k)$, by  $1.$, we obtain that $\theta_{ea}(k) \in \E(S)$. We get
\begin{align*}
\theta_{a^{-1}}\left(ea\right)   
&= \theta_{a^{-1}}\left(eakk^{-1}\right)
= \theta_{a^{-1}}\left(ea\right)
\theta_{a^{-1}ea}\left(kk^{-1}\right)
=\theta_{a^{-1}}\left(ea\right)
\theta_{e}\left(kk^{-1}\right)\\
&=\theta_{a^{-1}}\left(ea\right)
\theta_{e}\left(k\right)\theta_{ek}\left(k^{-1}\right).
\end{align*}
Hence, by \eqref{leq-prod}, $\theta_{a^{-1}}\left(ea\right)\leq  \theta_{e}\left(k\right)$. Therefore, the claim follows.
\end{proof}
\end{lemma}

\medskip

\begin{cor}
    Let $s$ be a solution on $S$. If $a,b\in S$ are such that $a\leq b$, then $\theta_a(b)\in G_{\theta_a\left(bb^{-1}\right)}$. Moreover, they hold $\theta_a(bb^{-1}) = \theta_a(b)\theta_a(b)^{-1}$ and $\theta_a(b)^{-1} = \theta_{ab}(b^{-1})$. 
    \begin{proof}
       If $a,b\in S$ are such that $a\leq b$, then  $a \leq bb^{-1}$ and by \cref{lemma_vicent}-$1.$, $\theta_a(bb^{-1})\in \E(S)$. Now,
$$
\theta_a(b) = \theta_a\left(bb^{-1}b\right) = \theta_a\left(bb^{-1}\right)\theta_{abb^{-1}}(b) = \theta_a\left(bb^{-1}\right)\theta_a(b)
$$
and $\theta_{a}\left(bb^{-1}\right) = \theta_a(b)\theta_{ab}\left(b^{-1}\right)$. Thus, by \eqref{leq-prod}, $\theta_a(b) \leq \theta_a(bb^{-1})$ and $\theta_a(bb^{-1})\leq \theta_a(b)$, i.e. $\theta_a(b)\in G_{\theta_a\left(bb^{-1}\right)}$.  In addition, by the equality $\theta_a\left(bb
^{-1}\right)=\theta_a\left(b^{-1}b
\right)=\theta_a\left(b^{-1}\right)\theta_{ab^{-1}}\left(b\right)$
and the previous paragraph, it follows that $\theta_a(b)$, $\theta_a(b^{-1})$, and $\theta_a(bb^{-1})$ are in the same group with identity $\theta_a(bb^{-1})$. Moreover, $\theta_a(b)^{-1} = \theta_{ab}\left(b^{-1}\right)$, which completes the proof.
    \end{proof}
\end{cor}

\medskip

\begin{theor}
    Let $s$ be a solution on $S$. Then, the kernel $K$ of $s$ is a normal subsemigroup of $S$.
    \begin{proof}
Initially, by \cref{lemma_vicent}-$1.$, $\E(S) \subseteq K$. Now, if $k,h \in K$ and $e \in \E(S)$ are such that $e \leq kh$, then $e \leq k$ and $e \leq h$ and thus, $\theta_e(k), \theta_e(h) \in \E(S)$.   By \cref{lemma_vicent}-$2.$, we obtain that $\theta_{ek}\left(h\right) \in \E(S)$, and so that $\theta_e\left(kh\right) = \theta_e\left(k\right)\theta_{ek}\left(h\right) \in \E(S)$. \\
Now, if $a \in S$, $k \in K$, and $e \in \E(S)$ are such that $e \leq a^{-1}ka$, then $e \leq a$, $e \leq a^{-1}$, and $e \leq k$. Then, $\theta_e(k) \in \E(S)$. Besides,
\begin{align*}
    \theta_e\left(a^{-1}ka\right)
    =
    \theta_e\left(a^{-1}\right)
    \theta_{ea^{-1}}(k)
    \theta_{ea^{-1}k}(a).
\end{align*}
By \cref{lemma_vicent}-$1.$, $\theta_e\left(a^{-1}\right) \in \E(S)$ and, by \cref{lemma_vicent}-$2.$, $\theta_{ea^{-1}}(k)\in \E(S)$.
Furthermore, also $\theta_{ea^{-1}k}(a)\in \E(S)$. In fact, by \eqref{eq:P2},
\begin{align*}
    \theta_{ea^{-1}k}\left(a\right)
    = \theta_{\theta_{k^{-1}a}\left(ea^{-1}k\right)}\theta_{k^{-1}aea^{-1}k}\left(a\right)
  = \theta_{\theta_{k^{-1}a}\left(ea^{-1}k\right)}\theta_{e}\left(a\right)
\end{align*}
and, since
\begin{align*}
    \theta_{k^{-1}a}\left(ea^{-1}k\right)
    &= \theta_{k^{-1}a}\left(ea^{-1}kaa^{-1}\right)
     \theta_{k^{-1}a}\left(ea^{-1}k\right)
    \theta_{k^{-1}aea^{-1}k}\left(aa^{-1}\right)\\
    &= \theta_{k^{-1}a}\left(ea^{-1}k\right)
    \theta_{e}\left(a\right)\theta_{ea}\left(a^{-1}\right),
\end{align*}
 we obtain that, by \eqref{leq-prod},
$\theta_{k^{-1}a}\left(ea^{-1}k\right)\leq \theta_{e}\left(a\right)$. So, as before, by \cref{lemma_vicent}-$1.$, we obtain
$\theta_{ea^{-1}k}\left(a\right)\in E\left(S\right)$.
Therefore, the claim follows.
    \end{proof}
\end{theor}

\medskip

We conclude the section by describing the commutative and cocommutative solutions on Clifford semigroups. It is easy to check that a solution $s(a,b)=(ab, \theta_a(b))$ is commutative if, and only if,
\begin{align}
    acb = abc\label{C1}\tag{C1}\\
   \theta_a = \theta_{ab}\label{C2}\tag{C2}
\end{align}
and $s$ is cocommutative if, and only if, 
\begin{align}
  \label{CC1}  &a\theta_b(c)=ac\tag{CC1}\\
  \label{CC2}  &\theta_a\theta_{b}=\theta_b\theta_a\tag{CC2}
\end{align} 
for all $a,b,c \in S$.
\begin{prop}
    Let $s$ be a solution on $S$. Then, 
    \begin{enumerate}
        \item $s$ is commutative if, and only if, $S$ is a commutative Clifford semigroup and $\theta_a=\gamma$, for every $a \in S$, with $\gamma \in \End(S)$ and $\gamma^2=\gamma$.
        \item $s$ is cocommutative if, and only if, $\theta_a(b) = b$, for all $a,b \in S$, i.e., $s=\mathcal{I}$.
    \end{enumerate} 
    \begin{proof}
    At first, we suppose that $s(a,b)=(ab, \theta_a(b))$ is a commutative solution. Then, by \eqref{C1}, taking $a=cc^{-1}$, we obtain that $S$ is commutative. Moreover, by \eqref{C2}, we get $\theta_a=\theta_{ab}=\theta_{ba}=\theta_b$. Hence, $\theta_a=\gamma$, for every $a \in S$, and by the definition of solution we obtain the rest of the claim. The converse trivially follows by $2.$ in \cref{exs_solutions}.\\
    Now, assume that $s(a,b)=(ab, \theta_a(b))$ is a cocommutative solution. Then, by \eqref{CC1}, taking $a=cc^{-1}$, we obtain
    \[ cc^{-1}\theta_b(c) = c, \quad \text{for all $b,c \in S$.}\]
    Set $e_0 := \theta_b(c)\theta_b(c)^{-1}$, it follows that $cc^{-1}\leq e_0$.
    On the other hand, again by \eqref{CC1}, $e\theta_b(c) = ec$, for every  $e \in \E(S)$. In particular, $\theta_b(c) = e_0 \theta_b(c) = e_0c$. Thus, $e_0 \leq cc^{-1}$ and so $e_0 = cc^{-1}$. Therefore, we get $\theta_b(c) = c$, that is our claim.
    \end{proof}
\end{prop}

\bigskip

\section{A description of idempotent-invariant solutions}
In this section, we provide a description of a specific class of solutions on a Clifford semigroup, the  idempotent-invariant ones, which includes the result contained in \cite[Theorem 15]{CaMaMi19}.

\medskip

\begin{defin}
A solution $s$  on $S$ is said to be \emph{idempotent-invariant} or \emph{$E\left(S\right)$-invariant} if it holds the identity
\begin{align}\label{eq:id-inv}
    \theta_a(e) = \theta_a(f),
\end{align}
for all $a\in S$ and $e,f\in\E(S)$. 
\end{defin}
\noindent An easy example of $\E(S)$-invariant solution is   $\mathcal{E}(a,b) = (ab, e)$ in \eqref{solu_Clifford}, with $e\in \E(S)$.

\begin{ex}\label{ex_monoide_in}
Let us consider the commutative Clifford monoid $S =\{1,\,a,\,b\}$ with identity $1$ and such that $a^2=a$, $b^2=a$, and $ab=b$. Then, other than the map $\mathcal{E}$ in \eqref{solu_Clifford}, there exists the idempotent-invariant solution $s(a,b)=(ab, \gamma(b))$ with $\gamma: S \to S$ the map given by $\gamma(1)=\gamma(a)=a$ and $\gamma(b)=b$, which belongs to the class of solutions in $2.$ of \cref{exs_solutions}.
\end{ex}
\medskip

Next, we show how to construct an idempotent-invariant solution on $S$ starting from a specific congruence on $S$. 
 Recall that the restriction of a congruence $\rho$ in a Clifford semigroup $S$ to $\E(S)$ is also a congruence on $\E(S)$,  called the \emph{trace} of $\rho$ and usually denoted by $\tau=\tr \rho$ (for more details, see \cite[Section 5.3]{Ho95}).

\begin{prop}\label{propmu}
	Let $S$ be a Clifford semigroup, $\rho$ a congruence on $S$ such that $S/\rho$ is a group, and $\mathcal{R}$ a system of representatives of $S/\rho$.  If $\mu: S \to \mathcal{R}$ is a map such that 
	\begin{equation}\label{cond_solu}
	 \mu\left(ab\right)=\mu\left(a\right)\mu\left(a\right)^{-1}\mu\left(ab\right),
	\end{equation}
	for all $a,b \in S$, and $\mu(a) \in [a]_{\rho}$, for every $a \in S$, then the map $s:S\times S\to S\times S$ given by 
	\begin{equation*}
	s(a,b)=\left(ab, \mu\left(a\right)^{-1}\mu\left(ab \right)\right),
	\end{equation*}
	for all $a,b \in S$, is an $\E(S)$-invariant solution on $S$.
\end{prop}
\begin{proof}
Let $a,b,c \in S$. Set $\theta_a(b):=\mu\left(a\right)^{-1}\mu\left(ab \right)$, by \eqref{cond_solu}, we obtain
\begin{align*}
\theta_a(b)\theta_{ab}(c)=\mu\left(a\right)^{-1}\mu\left(ab \right)\mu\left(ab\right)^{-1}\mu\left(abc \right)=\mu\left(a\right)^{-1}\mu\left(abc \right)=\theta_a(bc).
\end{align*}
Now, if we compare
	\begin{align*}
	\theta_{\theta_a(b)}\theta_{ab}(c):&= \mu\left(\mu\left(a\right)^{-1} \mu\left(a b\right)\right)^{-1}\mu\left(\mu\left(a\right)^{-1} \mu\left(ab\right)  \mu\left(ab\right)^{-1}  \mu\left(abc\right)\right)\\
	&=\mu\left(\mu\left(a\right)^{-1} \mu\left(ab\right)\right)^{-1} \mu\left(\mu\left(a\right)^{-1}  \mu\left(abc\right)\right)&\mbox{by \eqref{cond_solu}}
	\end{align*}
	and 	
		$\theta_b(c):=\mu(b)^{-1} \mu(bc)$,
to get the claim it is enough to show that
	\begin{align*}
	    \mu(x)^{-1}\mu(xy) = \mu(y),
	\end{align*}
	for all $x, y \in S$. Indeed, by \cite[Proposition 5.3.1]{Ho95}, $\tr \rho = \E(S) \times \E(S)$, and so
	\begin{align*}
	    \mu(x)^{-1}\mu(xy)
	      \ \rho \ 
	    x^{-1}xy  
	    \ \rho \  
	    y^{-1}yy 
	    \ \rho \
	    y
	    \ \rho \ 
	    \mu(y).
	\end{align*}
	Finally, if $a \in S$ and $e, f \in \E(S)$, we obtain that
\begin{align*}
    \mu(ae) \ \rho \ ae \ \rho \ af \rho \ \mu(af), 
\end{align*}
hence $\mu\left(a e\right)=\mu\left(a f\right)$. Thus, $\theta_a(e)=\mu(a)^{-1}\mu\left(a e\right)=\mu(a)^{-1}\mu\left(a f\right)=\theta_a(f)$.
	Therefore, the claim follows.
\end{proof}
\medskip

Our aim is to show that all  idempotent invariant solutions can be constructed exactly as in \cref{propmu}. Firstly, let us collect some useful properties of these maps.

\begin{lemma}\label{lemma_solu_fix_idemp}  
Let $s$ be an $E\left(S\right)$-invariant solution on $S$. Then, 
the following hold:
\begin{enumerate}
    \item $\theta_e=\theta_{f}$,
    \item $\theta_{ae} = \theta_a$, 
    \item $\theta_{a}\left(e\right)\in \E\left(S\right)$, 
    \item $\theta_e\theta_a = \theta_e$,
    \item $\theta_a\left(b\right)
    =  \theta_a\left(eb\right)$,
    \item $\theta_e(a)^{-1}=\theta_{ea}\left(a^{-1}\right)$,
\end{enumerate}
for all  $e,f \in \E\left(S\right)$ and $a, b\in S$.
\begin{proof}
Let $e, f \in \E(S)$ and $a, b\in S$.\\       
$1.$ \ Since $\theta_e = \theta_{\theta_f(e)}\theta_{fe} = \theta_{\theta_f(fe)}\theta_{ffe} = \theta_{fe}$ and, similarly $\theta_f=\theta_{ef}$, it yields that $\theta_f = \theta_e$.\\
$2.$ \ We have that
\begin{align*}
   \theta_{ae}&=\theta_{\theta_{a^{-1}}(ae)}\theta_{aa^{-1}e}\\
    &=\theta_{\theta_{a^{-1}}(a)\theta_{a^{-1}a}(e)}\theta_{aa^{-1}} &\mbox{$aa^{-1}e\in \E\left(S\right)$}\\
    &=\theta_{\theta_{a^{-1}}(a)\theta_{a^{-1}a}\left(a^{-1}a\right)}\theta_{aa^{-1}}&\mbox{by \eqref{eq:id-inv}}\\
    &=\theta_{\theta_{a^{-1}}\left(a\right)}\theta_{aa^{-1}}=\theta_a.
\end{align*}
$3.$ \ According to $2.$, it follows that
$\theta_{a}\left(e\right) 
    = \theta_{a}\left(ee\right)
    = \theta_{a}\left(e\right)\theta_{ae}\left(e\right)
    = \theta_{a}\left(e\right)\theta_{a}\left(e\right)$,
i.e., $\theta_{a}\left(e\right)\in\E\left(S\right)$.\\
$4.$ \ According to $2.$, we obtain that
$\theta_e = \theta_{\theta_{a}\left(e\right)}\theta_{ae}
    = \theta_{e}\theta_{ae}
    = \theta_{e}\theta_{a}$.\\
$5.$ \ Note that, by $2.$,
$\theta_a\left(b\right)
 = \theta_a\left(bb^{-1}b\right)
 = \theta_a\left(bb^{-1}\right)\theta_{abb^{-1}}\left(b\right)
 = \theta_a\left(e\right)\theta_{ae}\left(b\right)
 = \theta_a\left(eb\right)$.\\
 $6.$ \ Applying $1.$, we get $\theta_e\left(a\right)\theta_{ea}\left(a^{-1}\right)\theta_{e}(a)=\theta_e\left(aa^{-1}\right)\theta_{eaa^{-1}}(a)=\theta_e(a)$ and, on the other hand, $$\theta_{ea}\left(a^{-1}\right)\theta_e(a)\theta_{ea}\left(a^{-1}\right)=\theta_{ea}\left(a^{-1}\right)\theta_e\left(aa^{-1}\right)=\theta_{ea}\left(a^{-1}\right)\theta_{eaa^{-1}}\left(aa^{-1}\right)=\theta_e\left(a^{-1}\right).$$
Therefore, the claim follows.
\end{proof}
\end{lemma}

\medskip

To prove the converse of \cref{propmu}, we need to recall the notion of the congruence pair of inverse semigroups that are Clifford (see \cite[p. 155]{Ho95}). Given a Clifford semigroup $S$, a congruence $\tau$ on $\E(S)$ is said to be \emph{normal} if  
\begin{align*}
    \forall \ e,f\in \E(S) \quad e \ \tau \ f \ \Longrightarrow \ \forall \ a\in S \quad a^{-1}ea\ \tau \ a^{-1}fa.
\end{align*}
If $K$ is a normal subsemigroup of $S$, the pair $(K,\tau)$ is named a \emph{congruence pair} of $S$ if
\begin{align*}
  \forall\ a \in S,\ e\in \E(S)  \quad ae\in K \ \  \text{and} \ \ (e, a^{-1}a)\in \tau \ \Longrightarrow \ a\in K.
\end{align*}
Given a congruence $\rho$, denoted by $\Ker \rho$ the union of all the idempotent $\rho$-classes, its properties can be described entirely in terms of $\Ker \rho$ and $\tr \rho$.
\begin{theor}[cf. Theorem 5.3.3 in \cite{Ho95}]\label{theo_rho}
Let $S$ be an inverse semigroup. If $\rho$ is a congruence on $S$, then $(\Ker\rho, \tr\rho)$ is a congruence pair. Conversely, if $(K, \tau)$ is a congruence pair, then
\begin{align*}
    \rho_{(K, \tau)}=\lbrace (a,b) \in S \times S \, \mid \, \left(a^{-1}a, b^{-1}b\right) \in \tau, \, ab^{-1} \in K\rbrace
\end{align*}
is a congruence on $S$. Moreover, $\Ker \rho_{(K, \tau)}=K$, $\tr \rho_{(K, \tau)}=\tau$, and $\rho_{(\Ker \rho, \tr \rho)}=\rho$.
\end{theor}

\medskip

\begin{lemma}\label{le:cong-pair}  
Let $s$ be an $\E\left(S\right)$-invariant solution on $S$, $\tau = \E(S)\times\E(S)$, and $K$ the kernel of $s$. 
Then, $\left(K, \tau\right)$ is a congruence pair of $S$.
\begin{proof}
At first, let us observe that the kernel $K$ of $s$ can be written as
\begin{align*}
    K=\{a\in S\, \mid \, \forall \ e \in \E(S)\quad \theta_e(a) \in \E(S)\}.
\end{align*}
Now, let $a \in S$ and $e \in \E(S)$ such that $ae\in K$. To get the claim it is enough to show that if  $f\in\E(S)$, then $\theta_{f}\left(a\right)\in\E\left(S\right)$, i.e., $a\in K$.
By $1.$ and $5.$ in \cref{lemma_solu_fix_idemp}, we obtain that
\begin{align*}
    \theta_f\left(a\right)  
    = \theta_{ef}\left(a\right)
    = \theta_{ef}\left(ae\right)\in \E(S),
\end{align*}
 which is our claim.
\end{proof}
\end{lemma}

\medskip

The following result completely describes idempotent-invariant solutions.
\begin{theor}
Let $s$ be an $\E\left(S\right)$-invariant solution on $S$. Then,  the map $\theta_e$ satisfies \eqref{cond_solu}, for every $e \in \E(S)$, and
\begin{align*}
  \theta_a(b)=\theta_e(a)^{-1}\theta_e(ab), 
\end{align*}
for all $a, b\in S$ and $e\in\E\left(S\right)$. Moreover, there exists the congruence pair $\left(K, \tau\right)$, with $K$ the kernel of $S$ and $\tau=\E(S) \times \E(S)$, such that 
$\theta_e\left(S\right)$ is a system of representatives of the group $S/\rho_{\left(K, \tau\right)}$ and 
$\left(\theta_e\left(a\right), a \right) \in \rho_{\left(K, \tau\right)}$, for all $e \in \E(S)$ and $a\in S$.
\begin{proof}
Initially, \eqref{cond_solu} is satisfied since
\begin{align*}
    \theta_e(a)^{-1}\theta_e(a)\theta_e(ab)=\theta_e(a)^{-1}\theta_e(a)\theta_e(a)\theta_{ea}(b)=\theta_e(a)\theta_{ea}(b)=\theta_e(ab),
\end{align*}
for all $a, b\in S$ and $e\in\E\left(S\right)$. Besides,
\begin{align*}
    \theta_a(b)
    &=\theta_a\left(a^{-1}ab\right)
  &\mbox{by \cref{lemma_solu_fix_idemp}-$5.$}\\
    &=\theta_a\left( a^{-1}\right)\theta_{aa^{-1}}(ab)\\
    &=\theta_{aa^{-1}}(a)^{-1}\theta_{aa^{-1}}(ab), 
    &\mbox{by \cref{prop_properties_theta}-$1.$}\\
    &= \theta_e(a)^{-1}\theta_e(ab) &\mbox{by \cref{lemma_solu_fix_idemp}-$1.$}
\end{align*}
for all $a, b\in S$ and $e\in\E\left(S\right)$. Moreover, by \cref{le:cong-pair}, $\left(K, \tau\right)$ is a congruence pair and so, by \cref{theo_rho},  $\rho_{\left(K, \tau\right)}$ is a congruence such that $\tau=\tr \rho_{\left(K, \tau\right)}$. Besides, by \cite[Proposition 5.3.1]{Ho95}, since $\tr \rho_{\left(K, \tau\right)} = \E(S) \times \E(S)$, $S/\rho_{\left(K, \tau\right)}$ is a group.
Now, let $a\in S$ and $e\in\E\left(S\right)$ and let us check that $\left(\theta_e\left(a\right), a \right) \in \rho_{\left(K, \tau\right)}$ by proving that $a^{-1}\theta_e\left(a\right)\in K$, i.e., $\theta_e\left(a^{-1}\theta_e\left(a\right)\right)\in\E\left(S\right)$. To this end, note that
\begin{align*}
    \theta_e\left(a^{-1}\theta_e\left(a\right)\right)
    &= \theta_e\theta_a\left(a^{-1}\theta_e\left(a\right)\right)
    &\mbox{by \cref{lemma_solu_fix_idemp}-$4.$}\\
    &= \theta_e 
    \left(\theta_a\left(a^{-1}\right)
    \theta_{aa^{-1}}\theta_e\left(a\right)\right)\\
    &= \theta_e 
    \left(\theta_a\left(a^{-1}\right)
    \theta_{aa^{-1}}\left(a\right)\right)&\mbox{by \cref{lemma_solu_fix_idemp}-$4.$}\\
    &= \theta_e 
    \left(\theta_a\left(a^{-1}\right)
    \theta_{a}\left(a^{-1}\right)^{-1}\right),
    &\mbox{by \cref{prop_properties_theta}-$1.$}
\end{align*}
hence, by \cref{lemma_solu_fix_idemp}-$3.$, $\theta_e\left(a^{-1}\theta_e\left(a\right)\right)\in \E\left(S\right)$. 
Now, let us verify that $\theta_e\left(S\right)$ is a system of representatives of $S/\rho_{\left(K, \tau\right)}$. 
Clearly, $\theta_e\left(S\right)\neq\emptyset$ since $\theta_e\left(e\right)\in\E\left(S\right)$. Besides, if $\left(\theta_e\left(b\right), a \right) \in \rho_{\left(K, \tau\right)}$ we have that $a\,\rho_{\left(K, \tau\right)}\, b$, since $\left(\theta_e\left(a\right), a \right) \in \rho_{\left(K, \tau\right)}$. Thus, $ab^{-1}\in K$ and so $\theta_e\left(ab^{-1}\right)\in\E\left(S\right)$.
This implies that
\begin{align*}
\theta_e\left(b\right)  
&
= \theta_e\left(bb^{-1}\right)
\theta_{ebb^{-1}}\left(b\right)\\
&= \theta_e\left(bb^{-1}\right)
\theta_{\theta_e\left(ab^{-1}\right)}\left(b\right)&\mbox{by \cref{lemma_solu_fix_idemp}-$1.$}\\
&= \theta_e\theta_e\left(ab^{-1}\right)
\theta_{\theta_e\left(ab^{-1}\right)}\theta_{eab^{-1}}\left(b\right)&\mbox{by \eqref{eq:id-inv} and \cref{lemma_solu_fix_idemp}-$4.$}\\
&= \theta_e\left(ab^{-1}\right)
\theta_{ab^{-1}}\left(b\right)&\mbox{by \cref{lemma_solu_fix_idemp}-$4.$}\\
&= \theta_e\left(ab^{-1}\right)
\theta_{eab^{-1}}\left(b\right)&\mbox{by \cref{lemma_solu_fix_idemp}-$2.$ and \eqref{eq:P2}}\\
&= \theta_e\left(ab^{-1}b\right)\\
&= \theta_e\left(a\right).&\mbox{by \cref{lemma_solu_fix_idemp}-$5.$}
\end{align*}
Therefore, the claim follows.
\end{proof}
\end{theor}
\medskip

\begin{prop}
Let $s(a,b)=(ab, \theta_a(b))$ and $t(u,v)=(uv, \theta_u(v))$ be two $\E(S)$-invariant solutions on $S$. Then, $s$ and $t$ are isomorphic if, and only if, there exists an isomorphism $\psi$ of $S$ such that $\psi\theta_e=\eta_e\psi$, i.e., $\psi$ sends the system of representatives $\theta_e(S)$ into the other one $\eta_e\left(\psi(S)\right)$, for every $e \in \E(S)$.
\begin{proof}
Indeed, making explicit the condition \eqref{isomor}, we obtain \begin{align*}\psi\left(\theta_e(a)^{-1}\theta_e(ab)\right)=\eta_e\left(\psi(a)\right)^{-1}\eta_e\left(\psi(ab)\right),\end{align*}
for all $a,b \in S$ and $e \in \E(S)$. Applying \cref{lemma_solu_fix_idemp}-$6.$ and taking $b=a^{-1}$, we get $\psi\left(\theta_e(a)^{-1}\right)=\eta_e\left(\psi(a)\right)^{-1}$. Thus, the claim follows.
    \end{proof}
\end{prop}

\bigskip

\section{A construction of idempotent-fixed solutions}
In this section, we  deal with a class of solutions different from the idempotent-invariant ones, what we call idempotent-fixed solutions. Bearing in mind that a Clifford semigroup can be seen as a union  of groups satisfying certain properties, it is natural to contemplate  whether it is possible or not to construct a global solution in a Clifford semigroup from solutions obtained in each of its groups. In this regard, in the case of idempotent-fixed solutions, we manage to construct a family of solutions obtained by starting from given solutions on each group. 
\medskip

\begin{defin}
Let $s$ be a solution on $S$. Then, $s$ is \emph{idempotent-fixed} or \emph{$\E(S)$-fixed} if 
\begin{align}\label{fix_idem}
    \theta_a(e)=e,
\end{align}
for all $a \in S$ and $e \in \E(S)$.
\end{defin}
\noindent The maps $\mathcal{I}(a,b)=(ab,b)$ and $\mathcal{F}(a,b)=\left(ab, bb^{-1}\right)$ in \eqref{solu_Clifford} are idempotent-fixed solutions on $S$. Clearly, if $S$ is a Clifford that is not a group, i.e., $|\E(S)|>1$, then a solution on $S$ can not be both idempotent-fixed and idempotent-invariant.

\medskip

The next results contained several properties of idempotent-fixed solutions.
\begin{prop}
    Let $s$ be an idempotent-fixed solution on $S$. Then, $\theta_e=\theta_e\theta_{ae}$, for all $a \in S$ and $e \in \E(S)$. In particular, $\theta_e$ is an idempotent map.
    \begin{proof}
    It follows by $\theta_e=\theta_{\theta_a(e)}\theta_{ae}=\theta_e\theta_{ae}$, for all $a \in S$ and $e \in \E(S)$. Taking $a=e$, we obtain that the map $\theta_e$ is idempotent.
    \end{proof}
\end{prop}

\medskip
\begin{prop}\label{solu_fix_idemp}
    Let $s$ be an idempotent-fixed solution on $S$. Then, the following hold:
    \begin{enumerate}
        \item $\theta_a(b) = bb^{-1}\theta_a(b)$, 
        \item  $\theta_a\left(b\right)\theta_a\left(b\right)^{-1} = bb^{-1}$,
        \item $\theta_a(b) = \theta_{abb^{-1}}(b)$,
    \end{enumerate}
        for all $a,b \in S$.
        \begin{proof}
        Let $a,b\in S$. Then, 
        $\theta_a\left(b\right) = 
        \theta_a\left(b\right)\theta_{ab}\left(b^{-1}b\right)
        = \theta_{a}\left(b\right)bb^{-1}$.
        Moreover, we have that $\theta_a\left(b\right)^{-1} = \theta_{ab}\left(b^{-1}\right)$ since
        \begin{align*}
         \theta_a\left(b\right)\theta_{ab}\left(b^{-1}\right)\theta_a\left(b\right)
         = \theta_a\left(bb^{-1}\right)\theta_a\left(b\right)
         = bb^{-1}\theta_a\left(b\right)
         = \theta_a\left(b\right)
        \end{align*}
        and
        \begin{align*}
            \theta_{ab}\left(b^{-1}\right)\theta_a\left(b\right)\theta_{ab}\left(b^{-1}\right)
            &= \theta_{ab}\left(b^{-1}\right)\theta_a\left(bb^{-1}\right)
            = b^{-1}b\,\theta_{ab}\left(b^{-1}\right)
            = \theta_{ab}\left(bb^{-1}\right)\theta_{abb^{-1}b}\left(b^{-1}\right)\\
            &= \theta_{ab}\left(b^{-1}\right).
        \end{align*}
        It follows that
       $
         \theta_a\left(b\right)\theta_a\left(b\right)^{-1}
         = \theta_a\left(b\right)\theta_{ab}\left(b^{-1}\right)
         = \theta_a\left(bb^{-1}\right)
         = bb^{-1}.
       $
        Finally, by $1.$, we have that
        \begin{align*}
            \theta_{abb^{-1}}\left(b\right)
            = bb^{-1}\theta_{abb^{-1}}\left(b\right)
            = \theta_a\left(bb^{-1}\right) \theta_{abb^{-1}}\left(b\right) 
            = \theta_a\left(bb^{-1}b\right)
            = \theta_a\left(b\right)
        \end{align*}
        that completes the proof.
        \end{proof}
\end{prop}
\medskip

As a consequence of \cref{solu_fix_idemp}-$1.$, if $s$ is an idempotent-fixed solution on the Clifford $S$, it follows that every group in $S$ remains invariant by $\theta_a$, for all $a\in S$. Thus, motivated by the fact that solutions on groups are well-described, it makes sense to provide a method to construct this type of solutions from solutions on each group in $S$. To this end, the inner structure of a Clifford semigroup makes clear that conditions relating to different solutions on the groups of $S$ must be considered.  For instance, \cref{solu_fix_idemp}-$3.$ shows that $\theta_a\left(b\right) =
\theta_{\varphi_{e,f}\left(a\right)}\left(b\right)
$, for all $e,f\in \E(S)$, with $e\geq f$,  and all $a\in G_e$, $b\in G_f$. In light of these observations, we provide the following family of idempotent-fixed  solutions.

\begin{theor}\label{theor_costr_fix}
Let $s^{[e]}(a,b)=\left(ab,\, \theta^{[e]}_a\left(b\right)\right)$ be a solution on $G_e$, for every $e \in \E(S)$. Moreover, for all $e, f\in \E(S)$, let $\epsilon_{e, f}: G_{e} \to G_{f}$ be maps such that $\epsilon_{e, f} = \varphi_{e, f}$ if $e\geq f$. If the following conditions are satisfied
\begin{align}
   \label{1} \theta^{[h]}_{\epsilon_{ef, h}(ab)} &= \theta^{[h]}_{\epsilon_{e, h}(a)\epsilon_{f, h}(b)},\\
    \label{2} \epsilon_{f, h}\theta^{[f]}_{\epsilon_{e, f}(a)}(b) &= \theta^{[h]}_{\epsilon_{e, h}(a)}\epsilon_{f, h}(b),
\end{align}
for all $e, f, h\in \E(S)$ and $a\in G_{e}$ and $b \in G_{f}$, set $$\theta_a(b) := \theta^{[f]}_{\epsilon_{e, f}(a)}(b),$$ for all $a \in G_{e}$ and $b \in G_{f}$. Then, the map $s: S \times S \to S \times S$ given by $s(a,b)=(ab, \theta_a(b))$ is an idempotent-fixed solution on $S$.
\begin{proof}
Let $e,f,h\in \E(S)$, $a\in G_{e}$, $b\in G_{f}$, and $c\in G_{h}$. Then, since $s^{[fh]}$ is a solution on $G_{fh}$, we obtain
\begin{align*}
    \theta_a\left(bc\right)
    &= \theta_{a}\left(\varphi_{f, fh}\left(b\right)
    \varphi_{h, fh}\left(c\right)\right)
    = \theta^{[fh]}_{\epsilon_{e,fh}\left(a\right)}
    \left(\varphi_{f,fh}\left(b\right)\varphi_{h, fh}\left(c\right)\right)\\
    &= \theta^{[fh]}_{\epsilon_{e,fh}\left(a\right)}\varphi_{f,fh}\left(b\right)
    \theta^{[fh]}_{\epsilon_{e,fh}\left(a\right)\varphi_{f,fh}\left(b\right)}\varphi_{f,fh}\left(c\right).
\end{align*}
Besides, we have that
\begin{align*}
    \theta_a\left(b\right)\theta_{ab}\left(c\right)
    &= \theta^{[f]}_{\epsilon_{e,f}\left(a\right)}\left(b\right)
    \theta^{[h]}_{\epsilon_{ef,h}\left(ab\right)}\left(c\right)
    =
    \varphi_{f,fh}\theta^{[f]}_{\epsilon_{e,f}\left(a\right)}\left(b\right)
    \varphi_{h,fh}\theta^{[h]}_{\epsilon_{ef,h}\left(ab\right)}\left(c\right).
\end{align*}
Hence, noting that, by \eqref{1},  
\begin{align*}
 \theta^{[fh]}_{\epsilon_{e,fh}\left(a\right)}\varphi_{f,fh}\left(b\right) = 
\theta^{[fh]}_{\epsilon_{e,fh}\left(a\right)}\epsilon_{f,fh}\left(b\right) =  
\epsilon_{f,fh}\theta^{[f]}_{\epsilon_{e,f}\left(a\right)}\left(b\right) = \varphi_{f,fh}\theta^{[f]}_{\epsilon_{e,f}\left(a\right)}\left(b\right)   
\end{align*}
and
\begin{align*}
    \theta^{[fh]}_{\epsilon_{e,fh}\left(a\right)\varphi_{f,fh}\left(b\right)}\varphi_{f,fh}\left(c\right) 
    &= \theta^{[fh]}_{\epsilon_{e,fh}\left(a\right)\epsilon_{f,fh}\left(b\right)}\epsilon_{f,fh}\left(c\right) \\
    &= \theta^{[fh]}_{\epsilon_{ef,fh}\left(ab\right)}\epsilon_{f,fh}\left(c\right)&\mbox{by \eqref{1}}\\
    &= \epsilon_{h,fh}\theta^{[h]}_{\epsilon_{ef,h}\left(ab\right)}\left(c\right) &\mbox{by \eqref{2}}\\
    &= \varphi_{h,fh}\theta^{[h]}_{\epsilon_{ef,h}\left(ab\right)}\left(c\right),
\end{align*}
it follows that \eqref{eq:P1} is satisfied.
In addition, 
\begin{align*}
    \theta_{\theta_a\left(b\right)}\theta_{ab}\left(c\right)
    &= \theta_{\theta^{[f]}_{\epsilon_{e,f}\left(a\right)}\left(b\right)}\theta^{[h]}_{\epsilon_{ef, h}\left(ab\right)}\left(c\right)\\
    &= \theta^{[h]}_{\epsilon_{f,h}\theta^{[f]}_{\epsilon_{e,f}\left(a\right)}\left(b\right)}\theta^{[h]}_{\epsilon_{ef, h}\left(ab\right)}\left(c\right)\\
    &= \theta^{[h]}_{\theta^{[h]}_{\epsilon_{e,h}\left(a\right)}\epsilon_{f,h}\left(b\right)}
    \theta^{[h]}_{\epsilon_{e,h}\left(a\right)\epsilon_{f,h}\left(b\right)}\left(c\right) &\mbox{by \eqref{2} and \eqref{1}}\\
    &= \theta^{[h]}_{\epsilon_{f,h}\left(b\right)}\left(c\right)&\mbox{$s^{[h]}$ is a solution on $G_h$}\\
    &= \theta_b\left(c\right),
\end{align*}
thus \eqref{eq:P2} holds.
Finally, by \cite[Lemma 11-$1.$]{CaMaMi19}, $\theta_a(f)=\theta^{[f]}_{\epsilon_{e,f}(a)}(f)=f$ and so $s$ is idempotent-fixed.
\end{proof}
\end{theor}
\medskip

The following is a class of idempotent-fixed solutions on $S$ that can be constructed through \cref{theor_costr_fix} and includes the solutions $\mathcal{I}(a,b)=(ab,b)$ and $\mathcal{F}(a,b)=\left(ab, bb^{-1}\right)$  in \eqref{solu_Clifford}.
\begin{ex}
Let $s^{[e]}\left(a,b\right) = \left(ab, \gamma^{[e]}\left(b\right)\right)$ be the solution on $G_e$ as in $2.$ of \cref{exs_solutions} with $\gamma^{[e]}$ an idempotent endomorphism of $G_e$, for every $e\in \E(S)$. Then, by choosing maps $\epsilon_{e,f}:G_e\to G_f$, for all $e,f\in \E(S)$ such that $\varphi_{e,f}\gamma^{[e]} = \gamma^{[f]}\varphi_{e,f}$ if $e\geq f$ and $\epsilon_{e,f}\left(x\right) := f$ otherwise,  then conditions \eqref{1} and \eqref{2} are satisfied. Hence, the map
\begin{align*}
 s(a,b)=\left(ab, \gamma^{[f]}(b)\right),   
\end{align*}
for all $a \in G_e$ and $b \in G_f$, is a solution on $S$.
  
\end{ex}
\medskip

As a consequence of \cref{theor_costr_fix}, the following construction provides a subclass of idempotent-fixed solutions in Clifford semigroups in which each group $G_f$ is an epimorphic image of  $G_e$, whenever $f \leq e$, for all $e,f\in \E(S)$.

\begin{cor}
Let $S$ be a Clifford semigroup such that $\varphi_{e,f}$ is an epimorphism, for all $e,f \in \E(S)$ with $f \leq e$. Let $s^{[e]}(a,b) = \left(ab, \theta_a^{[e]}(b)\right)$ be a solution on $G_e$ and set $N_e := \prod\limits_{f\leq e} \ker\varphi_{e,f}$, for every $e\in \E(S)$. Suppose that 
\begin{enumerate}
    \item $\theta_a^{[e]} = \theta_b^{[e]}$, for all $e\in \E(S)$ and all $a,b \in G_e$ with $aN_e = bN_e$,
    \item 
    $\varphi_{e,f}\theta_a^{[e]}(b) = \theta_{\varphi_{e,f}(a)}^{[f]}\varphi_{e,f}(b)$, for all $e,f\in \E(S)$ with $f \leq e$, and all $a,b\in G_e$.
\end{enumerate}
Set $\theta_a(b) := \theta_{b'}^{[f]}(b)$, with $b'\in G_{f}$ such that $\varphi_{f,ef}(b) = \varphi_{e,ef}(a)$, for all $e,f\in \E(S)$, and all $a\in G_e$, $b \in G_f$. Then, the map $s\colon S\times S \rightarrow S \times S$ given by $s(a,b) = (ab,\theta_a(b))$ is an idempotent-fixed  solution on $S$.

\begin{proof}
Initially, by $1.$, note that $\theta_a$ is well-defined, for every $a\in S$. Now, let $e,f\in \E(S)$ and consider $T_{e,f}$ a system of representatives of $\ker\varphi_{f,ef}$ in $G_f$. Since $\varphi_{f,ef}$ is an epimorphism, for every $a \in G_e$, we can define a map $\epsilon_{e,f}(a) := x \in T_{e,f}$, with $\varphi_{e,ef}(a) = \varphi_{f,ef}(x)$. Specifically, in the case that $f \leq e$, it follows that $\epsilon_{e,f} = \varphi_{e,f}$. Therefore, for all $e,f\in \E(S)$ and all $a\in G_e$, $b\in G_f$, it holds  $\theta_a(b) = \theta_{\epsilon_{e,f}(a)}^{[f]}(b)$. Note that, by $1.$, the last equality is independent of the choice of $T_{e,f}$. Moreover, applying properties in \cref{theo_Clifford} of  homomorphisms $\varphi_{e,f}$, for all $e,f\in \E(S)$ with $f\leq e$,  and the assumptions, it is a routine computation to check that conditions~\eqref{1} and~\eqref{2} of \cref{theor_costr_fix} are satisfied.

\end{proof}

\end{cor}

\medskip

Let us observe that the kernel of an idempotent-fixed solution $s$ can be rewritten as
\begin{align*}
    K=\{a \in S \, \mid \, \forall\, e \in \E(S), \,  e \leq a, \ \theta_e(a)=aa^{-1}\}.
\end{align*}
Denoted by $K_e$ the kernel of each solution $s^{[e]}$ on $G_e$,  i.e., the normal subgroup
\begin{align*}
    K_e=\{a \in G_e \, \mid \, \theta^{[e]}_{e}(a)=e\}.
\end{align*}
of $G_e$, we have the following result that clarifies the previous construction in \cref{theor_costr_fix} is not a description.
\begin{prop}
    Let $s$ be an idempotent-fixed solution on $S$ constructed as in \cref{theor_costr_fix} and suppose that $\epsilon_{e,f}(e)=f$, for all $e,f \in \E(S)$ with $e \leq f$. Assume that each $G_e$ admits a solution $s^{[e]}$ and let $K_{e}$ be the kernel of such a map  $s^{[e]}$, for every $e \in \E(S)$.   Then, $K = \displaystyle \bigcup_{e\in \E(S)}\,K_{e}$.
    \begin{proof}
    Indeed, let $a \in K \cap G_e$. Then, we get
    $ e = aa^{-1} = \theta_e(a) = \theta^{[e]}_e(a)$.
    Thus,  $a \in K_e$.  On the other hand, if $a \in K_e$ and $f \in \E(S)$ is such that $f \leq a$, then, since $\epsilon_{e,f}(e)=f$, we obtain $\theta_f(a)=\theta^{[e]}_{\epsilon_{f,e}(f)}(a)=\theta^{[e]}_{e}(a)=e$, i.e., $a \in K$.
    \end{proof}
\end{prop}
\medskip

\begin{que}
Complete a description of all the idempotent-fixed solutions.
\end{que}

\bigskip

To conclude, we observe that not every solution on $S$ lies in the class of idempotent invariant or idempotent-fixed solutions. Indeed, even in Clifford semigroups of low order, it is possible to construct such an example.

\begin{ex}\label{ex_out}
    Let $S=\{1,\,a,\,b\}$ be the Clifford monoid in \cref{ex_monoide_in}. Then, the maps 
    \begin{align*}
       & \theta_1(x)=a, \quad\text{for every $x$} \in S, \\
       & \theta_a=\theta_b: S \to S, \quad \text{given by} \,\,\theta_a(1)=1, \, \,\theta_a(a)=\theta_a(b)=a
    \end{align*}
    give rise to a solution on $S$ that is neither idempotent invariant, nor idempotent fixed.
\end{ex}
\medskip

\begin{que}
Study other classes of solutions on Clifford semigroups, including, for instance, the map in \cref{ex_out}.
\end{que}
\bigskip

\bibliography{bibliography}

\end{document}